\documentclass[12pt]{article}
\usepackage{amssymb, amscd, amsthm, amsmath,latexsym, amstext,mathrsfs,verbatim}
\usepackage[all]{xy}
\usepackage{graphicx,enumerate, url}
\def\lijntje{\vrule height2.4pt depth-2pt width0.5in}

\def\vlijntje{\vrule height0.45in depth0.4pt width0.4pt}
\def\vlijn{\buildrel {\hbox to 0pt{\hss$\textstyle\circ$\hss}}\over\vlijntje}

\def\dlijntje{{\vrule height2pt depth-1.6pt
width0.5in}\llap{\vrule height4pt depth-3.6pt width0.5in}}

\def\vtriple#1\over#2\over#3{\mathrel{\mathop{\kern0pt #2}\limits_{\hbox
to 0pt{\hss$#1$\hss}}^{\hbox to 0pt{\hss$#3$\hss}}}}
\def\rvtriple#1\over#2\over#3{\mathrel{\mathop{\kern0pt #2}\limits_{\hbox
to 0pt{\hss$#3$\hss}}^{\hbox to 0pt{\hss$#1$\hss}}}}
\def\Ct{\vtriple{\scriptstyle 2}\over\circ\over{}
\kern-1pt\lijntje\kern-1pt\vtriple{\scriptstyle 1}\over\circ\over{}
\kern-4pt{\dlijntje \kern -25pt<}\kern8pt
\vtriple{\scriptstyle 0}\over\circ\over{}\kern-1pt
}
\def\Bt{\vtriple{\scriptstyle 2}\over\circ\over{}
\kern-1pt\lijntje\kern-1pt\vtriple{\scriptstyle 1}\over\circ\over{}
\kern-4pt{\dlijntje \kern -25pt>}\kern8pt
\vtriple{\scriptstyle 0}\over\circ\over{}\kern-1pt}

\def\alg{{A}}

\def\ddA{{\rm A}}
\def\ddD{{\rm D}}

\def\Br{{\rm Br}}
\def\SBr{{\rm SBr}}

\def\BrM{{\rm BrM}}

\def\ddB{{\rm B}}
\def\ddC{{\rm C}}
\def\ddD{{\rm D}}
\def\ddE{{\rm E}}
\def\ddF{{\rm F}}

\def\hE{{\hat E}}

\newcommand{\cA}{\mathcal{A}}

\newcommand{\R}{\mathbb R}
\newcommand{\Z}{\mathbb Z}

\newcommand{\fp}{\mathfrak{p}}

\numberwithin{equation}{section}

\newtheorem{lemma}{Lemma}[section]
\newtheorem{cor}[lemma]{Corollary}
\newtheorem{prop}[lemma]{Proposition}
\newtheorem{thm}[lemma]{Theorem}

\theoremstyle{definition}

\newtheorem{defn}[lemma]{Definition}

\theoremstyle{remark}
\newtheorem{rem}[lemma]{Remark}

\begin{document}
\title{Brauer algebras of type $\ddF_4$}
\author{Shoumin Liu}
\date{}
\maketitle

\begin{abstract}
We  present an algebra  related to the Coxeter group
of type $\ddF_4$ which can be  viewed as the Brauer algebra of type $\ddF_4$ and is obtained as a subalgebra of the
Brauer algebra of type $\ddE_6$. We also describe some properties
of this algebra.
\end{abstract}

\begin{section}{Introduction}
When studying tensor decompositions for orthogonal groups,
Brauer  (\cite{Brauer1937}) introduce algebras which we now call   Brauer algebras of type $\ddA$.
 Cohen, Frenk and Wales (\cite{CFW2008}) extended  the
definition to simply laced types, including type $\ddE_6$.
Tits (\cite{T1959}) described how to obtain the  Coxeter group of type $\ddF_4$ as the fixed subgroup of  the  Coxeter group of  type $\ddE_6$
under a diagram automorphism (also seen in \cite{Car}). For this, M\"uhlherr gave a more general way by admissible partitions to obtain Coxeter groups as  subgroups in
 Coxeter groups in \cite{Muehl92}.
Here we will apply a similar  method   to the Brauer algebra $\Br(\ddE_6)$.
This is a part of a project to define Brauer algebras of spherical types (\cite{CLY2010}, \cite{CLY2011}). It turns out
that the presentation by generators and relations obtainable from the Dynkin diagram of type $\ddF_4$ in the same way as was done  for type $\ddB_n$ (\cite{CLY2011}) and $\ddC_n$ (\cite{CLY2010}).
\\First we give the definition
of $\Br(\ddF_4)$ using a presentation. Let $\delta$ be the generator of  the infinite cyclic group.
\begin{defn}\label{0.1}
The Brauer algebra of type $\ddF_4$, denoted by $\Br(\ddF_4)$,
is a unital associative $\Z[\delta^{\pm 1}]$-algebra generated by $\{r_i,\, e_i\}_{i=1}^{4}$,  subject to
the following relations.
\begin{eqnarray}
r_{i}^{2}&=&1 \qquad \qquad\,\,\,\kern.02em \mbox{for}\,\mbox{any} \ i   \label{0.1.3}
\\
r_ie_i &= & e_ir_i \,=\, e_i \,\,\,\,\,\,\kern.05em \mbox{for}\,\mbox{any}\ i  \label{0.1.4}
\\
e_{i}^{2}&=&\delta e_{i} \qquad \quad\,\,\kern.02em \mbox{for}\ i>2     \label{0.1.5}
\\
e_{i}^{2}&=&\delta^2 e_{i} \qquad \quad\,\,\kern.02em \mbox{for}\ i<3              \label{0.1.6}
\\
r_ir_j&=&r_jr_i, \qquad \quad \mbox{for}\ i\nsim j   \label{0.1.7}
\\
e_ir_j&=&r_je_i, \qquad  \quad \kern-.03em \mbox{for}\ i\nsim j     \label{0.1.8}
\\
e_ie_j&=&e_je_i, \qquad \quad \kern-.06em \mbox{for}\ i\nsim j      \label{0.1.9}
\\
r_ir_jr_i&=&r_jr_ir_j, \qquad\,\kern-.04em \mbox{for}\ {i\sim j}\,    \label{0.1.10}
\\
r_jr_ie_j&=&e_ie_j , \quad \qquad \kern-.11em \mbox{for}\ i\sim j              \label{0.1.11}
\\
r_ie_jr_i&=&r_je_ir_j, \quad \quad \kern.06em \mbox{for}\ i\sim j        \label{0.1.12}
\end{eqnarray}
and for  $\vtriple{\scriptstyle 2}\over\circ\over{}
\kern-4pt{\dlijntje \kern -25pt<}\kern8pt
\vtriple{\scriptstyle 3}\over\circ\over{}\kern-1pt$ ,
\begin{eqnarray}
r_2r_3r_2r_3&=&r_3r_2r_3r_2                                       \label{0.1.13}
 \\
r_2r_3e_2&=&r_3e_2                                \label{0.1.14}
\\
  r_2e_3r_2e_3&=&e_3e_2e_3                                                        \label{0.1.15}
\\
(r_2r_3r_2)e_3&=&e_3(r_2r_3r_2)                                                            \label{0.1.16}
\\
e_2r_3e_2&=&\delta e_2                                                \label{0.1.17}
\\
e_2e_3e_2&=&\delta e_2                                                    \label{0.1.18}
\\
e_2r_3 r_2&=&e_2r_3                                                          \label{0.1.19}
\\
 e_2e_3r_2&=&e_2e_3                                                         \label{0.1.20}
\end{eqnarray}
Here $i\sim j$ means that $i$ and $j$ are connected by a simple bond and $i\nsim j$ means that
there is no bond (simple or multiple) between $i$ and $j$ in the Dynkin Diagram of type $\ddF_4$ depicted in the Figure \ref{E6F4D}.
 The submonoid
of the multiplicative monoid of $\Br(\ddF_4)$
generated by $\delta$, $\{r_i,\,e_i\}_{i=1}^{4}$ is
denoted by $\BrM(\ddF_4)$. This is the monoid of monomials in
$\Br(\ddF_4)$.
\end{defn}
The defining relations (\ref{0.1.13})--(\ref{0.1.20}) can be found in $\Br(\ddC_2)$ in \cite{CLY2010} and
$\Br(\ddB_2)$ in \cite{CLY2011} by renumbering indices. Note that these relations are not symmetric for $2$ and $3$.
Their relations are fully determined by the  Dynkin diagram in the sense that all relations depend only on the vertices
and bonds of the Dynkin diagram and the lengths of their roots.   \\
It is well known that the Coxeter group $W(\ddF_4)$ of type $\ddF_4$, can be obtained  as the subgroup from
the Coxeter group  $W(\ddE_6)$ of type $\ddE_6$, of elements invariant  under the automorphism
of $W(\ddE_6)$ determined by the diagram automorphism $\sigma=(1,6)(3,5)$ indicated as a permutation on
the generators of $W(\ddE_6)$ whose Dynkin diagram are labeled and presented in  Figure \ref{E6F4D}.
\begin{figure}[!htb]
\begin{center}
\includegraphics[width=.7\textwidth,height=.4\textheight]{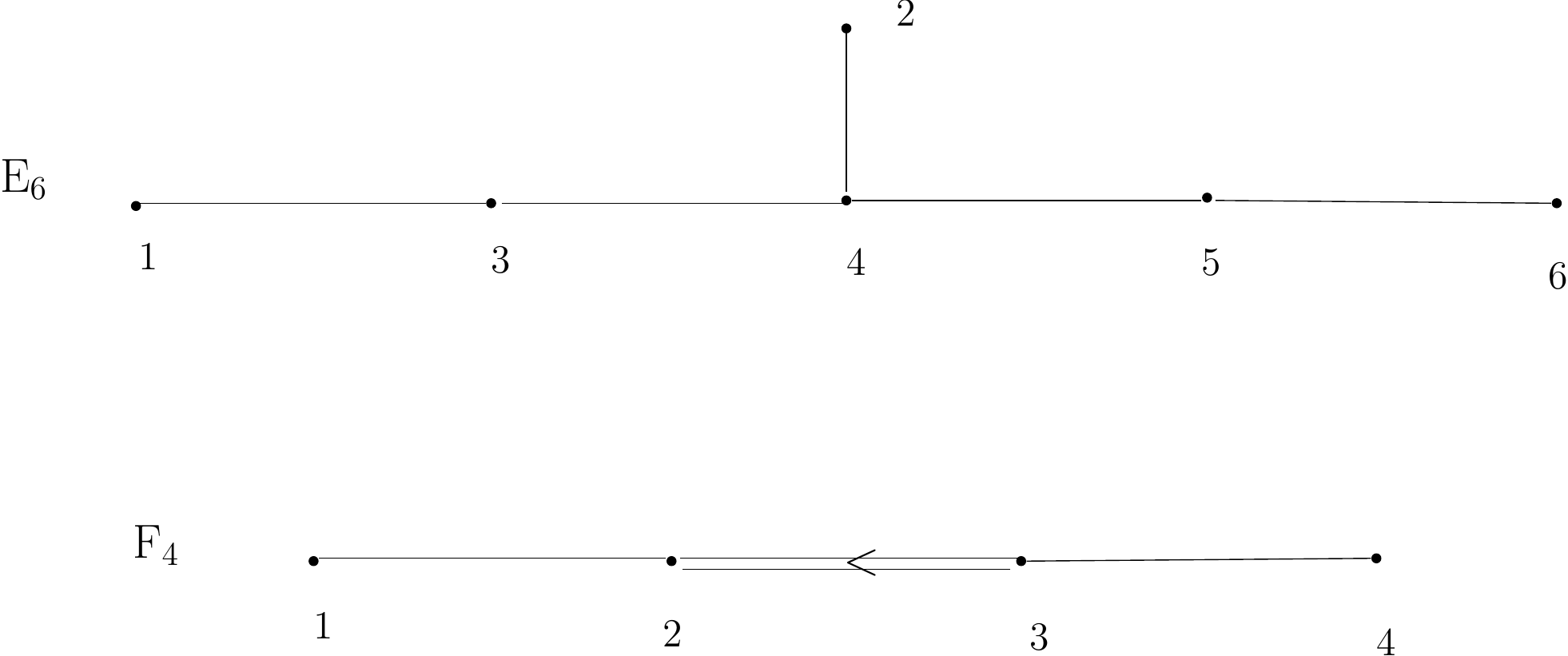}
\end{center}
\caption{Dynkin diagrams of $\ddE_6$ and $\ddF_4$}
\label{E6F4D}
\end{figure}

The action $\sigma$ can be extended to an automorphism
 of the Brauer algebra of type $\ddE_6$ by acting on the Temperley-Lieb generators $E_i$ (\cite{TL1971}) by the same permutation as for
 Weyl group generators.
We denote by $\SBr(\ddE_6)$  the subalgebra generated by  $\sigma$-invariant elements
in $\BrM(\ddE_6)$.
The main theorem of this paper is the following. In order to avoid confusion
with the above generators, the generators of $\Br(\ddE_6)$ have been capitalized.

\begin{thm} \label{mainthm}
There is an  algebra isomorphism
$$\phi:\, \Br(\ddF_4)\longrightarrow \SBr(\ddE_6)$$
determined by $\phi(r_1)=R_1R_6$, $\phi(r_2)=R_3R_5$, $\phi(r_3)=R_4$, $\phi(r_4)=R_2$,
and $\phi(e_1)=E_1E_6$, $\phi(e_2)=E_3E_5$, $\phi(e_3)=E_4$, $\phi(e_4)=E_2$.
Furthermore, the algebra $\Br(\ddF_4)$  is free over $\Z[\delta^{\pm 1}]$ of rank
$14985$.
\end{thm}
\end{section}
The proof of the theorem is finished in Section \ref{sect:image}.
Moreover, in the last section we will prove that the algebra $\Br(\ddF_4)\otimes R$ for a field $R$ is cellularly stratified.
This paper is included as chapter $4$ in the author's PhD thesis (\cite{L2012}). 

\section{Basic properties of $\Br(\ddF_4)$}
By the properties of $\Br(\ddB_3)$ in \cite{CLY2011} or $\Br(\ddC_3)$ in \cite{CLY2010},
we have more relations between $\{r_2,\, r_3,\, e_2, \, e_3\}$ Such as those of \cite[Lemma 4.1]{CLY2010}.
Just as \cite[Remark 3.5]{CLY2010} for type $\ddC$, there is  an anti-involution on $\Br(\ddF_4)$.

\begin{prop} \label{prop:opp} There is a unique  anti-involution on $\Br(\ddF_4)$ that fixes the generators $r_i$, $e_i$ $(1\leq i\leq 4)$.
\end{prop}
\begin{defn}\label{1.1}
Let $Q$ be a graph. The Brauer monoid $\BrM(Q)$ is the monoid
generated by the symbols $R_i$ and $E_i$, for  each node $i$ of $Q$ and $\delta$,
$\delta^{-1}$ subject to the  relation (\ref{0.1.3})--(\ref{0.1.5}) and (\ref{0.1.7})--(\ref{0.1.12}).
The Brauer algebra $\Br(Q)$ is   the free $\Z[\delta^{\pm 1}]$-algebra for Brauer monoid  $\BrM(Q)$.
\end{defn}
\begin{prop} \label{phibeinghomo}
The map $\phi$  determined on generators in Theorem \ref{mainthm} induces an algebra homomorphism from
$\Br(\ddF_4)$ to $\Br(\ddE_6)$.
\end{prop}
\begin{proof}It suffices to prove that the relations in Definition \ref{0.1} still holds when the generators are replaced by their images under $\phi$.
 The difficult ones are (\ref{0.1.13})--(\ref{0.1.20});
  which have been proved in the meanwhile of  the homomorphism from $\Br(\ddC_2)$ to $\Br(\ddA_3)$ in \cite{CLY2010}.
\end{proof}
To  distinguish them  from  the generators of $\Br(\ddF_4)$, we denote  the generators of $\Br(\ddC_3)$
($\Ct$) as
in \cite{CLY2010} by $\{r'_i, e'_i\}_{i=0}^{2}$ and  the generators of $\Br(\ddB_3)$ ($\Bt$)
 in \cite{CLY2011} by  $\{r''_i, e''_i\}_{i=0}^{2}$.
By checking their defining relations, we have the following proposition.
\begin{prop}\label{C3B3inF4} There are  injective algebra homomorphisms
$$\phi_1:\quad\Br(\ddC_3)\rightarrow\Br(\ddF_4)$$ 
$$\phi_2:\quad\Br(\ddB_3)\rightarrow\Br(\ddF_4),$$ 
 defined on generators as follows.
\begin{eqnarray*}
\phi_1(r'_i)&=&r_{3-i},\,\phi_1(e'_i)=e_{3-i}, \, \quad\mbox{for} \quad\, 0\leq i\leq 2,  \\
\phi_2(r''_i)&=&r_{2+i},\,\phi_2(e''_i)=e_{2+i}, \quad` \mbox{for} \quad \, 0\leq i\leq 2.
\end{eqnarray*}
\end{prop}
\begin{proof}
Just checking the defining relations of two algebras, we find that  $\phi_1$ and $\phi_2$ are algebra morphisms.
We see that $\phi\phi_1(\Br(\ddC_3))$ is contained in the subalgebra of $\Br(\ddE_6)$ generated
by $\{R_1,\,E_1\}\cup\{R_i,\,E_i\}_{i=3}^{6}$, which is isomorphic to $\Br(\ddA_5)$ according to \cite{CW2011}. By the main theorem of
\cite{CLY2010} and the behavior of $\phi\phi_1$ on generators, the homomorphism $\phi\phi_1$ coincides with the embedding of $\Br(\ddC_3)$ in $\Br(\ddA_5)$.
  Similarly, with \cite{CW2011} and the main theorem of \cite{CLY2011}, the homomorphism
$\phi\phi_2$ coincides with the embedding of $\Br(\ddB_3)$ in $\Br(\ddD_4)$; by application of these two theorems, it follows that the homomorphisms $\phi_1$ and $\phi_2$ are injective.
\end{proof}

\section{ Admissible root sets}
In this section, we  give the definition and description for admissible sets of type $\ddF_4$
and study some of their basic properties.\\
Let $\{\beta_i\}_{i=1}^4$ be simple roots of $W(\ddF_4)$.
They can be realized in $\R^4$ as
\begin{eqnarray*}
\beta_1&=&\frac{\epsilon_1-\epsilon_2-\epsilon_3-\epsilon_4}{2}, \, \beta_2=\epsilon_2,\\
\beta_3&=&\epsilon_3-\epsilon_2,\, \beta_4=\epsilon_4-\epsilon_3,
\end{eqnarray*}
with $\{\epsilon_i\}_{i=1}^{4}$ being the standard orthonormal basis of $\R^4$.
The set
$$\Psi^+=\{\frac{\epsilon_1\pm \epsilon_2\pm\epsilon_3\pm\epsilon_4}{2} \}
\cup\{\epsilon_i\}_{i=1}^{4}\cup \{\epsilon_j\pm\epsilon_i\}_{1\leq i<j\leq 4}$$
of cardinality $24$
is the  set of   positive roots of a root system $\Psi$ of  $W(\ddF_4)$ having $\beta_1$, $\ldots$, $\beta_4$ as simple roots.
We call a vector  $\beta\in \Psi^+$ a short root if its Euclidean length is $1$,
a long root if  its  Euclidean length is $\sqrt{2}$.\\
Let   $\{\alpha_i\}_{i=1}^6$ be simple roots of $W(\ddE_6)$.
The  $\{\alpha_i\}_{i=1}^6$  span a linear space over $\R$  of dimension $6$.
We define a linear map $\fp:\R^6\rightarrow \R^4$ by specifying its images on the given basis : \\
$\fp(\alpha_1)=\beta_1$,  $\fp(\alpha_6)=\beta_1$, $\fp(\alpha_3)=\beta_2$,
$\fp(\alpha_5)=\beta_2$,  $\fp(\alpha_4)=\beta_3$,  $\fp(\alpha_2)=\beta_4$.\\
Now  $\fp(\R^6)$ is  the $\sigma$-invariant space of $\R^6$, where $\sigma$ is the linear transformation of $\R^6$ determined by:\\
 $\sigma(\alpha_1)=\alpha_6$,  $\sigma(\alpha_6)=\alpha_1$, $\sigma(\alpha_3)=\alpha_5$,
$\sigma(\alpha_5)=\alpha_3$,  $\sigma(\alpha_4)=\alpha_4$,  $\sigma(\alpha_2)=\alpha_2$.\\
Let $\Phi\subset \R^6$ be the root system of $\ddE_6$ with simple roots $\{\alpha_i\}_{i=1}^{6}$, and
$\Phi^+$  the positive roots of $\Phi$.
In \cite{CFW2008}, the admissible root sets for simply-laced type are given. Now we will define the admissible
for type $\ddF_4$.
\begin{defn} Let $X\subset\Psi^+$ be a set of  mutually orthogonal roots. The set $X$ is called \it{admissible}
if $\fp^{-1}(X)\cap \Phi^+$ is  an \it{admissible} set.
\end{defn}

\begin{prop}\label{prop: admset}
There is a one-to-one correspondence between $\sigma$-invariant admissible root sets of type $\ddE_6$ and
the admissible root sets of type $\ddF_4$.
 Collections of all  admissible sets can be partitioned into  six $W(\ddF_4)$-orbits given by the following representatives.
\begin{enumerate}[(I)]
\item $\emptyset$,
\item  $\{\epsilon_4-\epsilon_3\}$,
\item  $\{\frac{\epsilon_1-\epsilon_2-\epsilon_3-\epsilon_4}{2}\}$,
\item  $\{\epsilon_3-\epsilon_2,\epsilon_3+\epsilon_2\}$,
\item $\{\frac{\epsilon_1-\epsilon_2-\epsilon_3-\epsilon_4}{2}, \epsilon_3-\epsilon_2, \epsilon_4+\epsilon_1\}$,
\item  $\{\epsilon_3-\epsilon_2,\epsilon_3+\epsilon_2, \epsilon_4+\epsilon_1, \epsilon_4-\epsilon_1\}$.
\end{enumerate}
Furthermore, their cardinalities are respectively, $1$, $12$, $12$, $18$, $36$, $3$.
\end{prop}
\begin{proof} In the admissible root sets of type $\ddE_6$, there are four $W(\ddE_6)$-orbits, which are
the orbits of
$\emptyset$, $\{\alpha_2\}$, $\{\alpha_3, \,\alpha_5\}$, $\{\alpha_2,\,\alpha_3,\,\alpha_5,\, 2\alpha_4+\alpha_2+\alpha_3+\alpha_5\}$ of
respective sizes.
 The $W(\ddE_6)$-orbits are easily seen to possess $1$ ((I)), $12$ ((II)), $30$ ((III), (IV)),
 $39$ ((V), (VI)) $\sigma$-invariant admissible root sets which can be decomposed into $W(\ddF_4)$-orbits  with representatives  (I), (II), (III), (IV),(V), (VI) as listed in the
 proposition.
\end{proof}

As in \cite{CFW2008}, \cite{CLY2010}, and \cite{CLY2011}, we have the following lemmas.
\begin{lemma}\label{lm:SimpleRootRels}
Let $i$ and $j$ be nodes of the Dynkin diagram $\ddF_4$.
If $w\in W(\ddF_4)$ satisfies $w\beta_i=\beta_j$, then
$we_iw^{-1}=e_j$.
\end{lemma}

Consider a positive root $\beta$ and a node $i$ of type $\ddF_4$.

If there
exists $w\in W$ such that $w\beta_i=\beta$, then  due to Lemma \ref{lm:SimpleRootRels} we can define the element
$e_{\beta}$ in $\BrM(\ddF_4)$ by
$$e_{\beta}=we_iw^{-1}.$$
 In general,
$$we_{\beta}w^{-1}=e_{w\beta},$$ for $w\in W(\ddF_4)$ and $\beta$ a root
of $W(\ddF_4)$, and  here we just consider the natural action of $W(\ddF_4)$ on $\Psi^+$ by negating negative roots.
 Analogous to the  argument in \cite[Lemma 4.5]{CLY2011}, the following lemma can be obtained by checking case by case listed in Proposition
 \ref{prop: admset}.
\begin{lemma}\label{lem:2roots}
 Let $\gamma_1$, $\gamma_2\in \Psi^+$ and $\gamma_1$ orthogonal to $\gamma_2$ such  that  $\{\gamma_1,\,\gamma_2\}$
 can be a subset of some admissible root set,
then $$e_{\gamma_1}e_{\gamma_2}=e_{\gamma_2}e_{\gamma_1}.$$
\end{lemma}

If $X\subset \Psi^+$ is a subset of  some admissible root set,  then by the lemma we can define
\begin{eqnarray}\label{e}
e_{X}=\Pi_{\beta\in X}e_{\beta}.
\end{eqnarray}
\begin{defn}\label{defn:admclo}
 Suppose that $X\subset \Psi^+$ is a mutually orthogonal root set.
 If $X$ can be contained in some admissible root set, Then the  minimal admissible set containing $X$
 is called the $admissible$ $closure$ of $X$, denoted by $X^{\rm cl}$.
\end{defn}
Thanks to an argument similar to \cite[Lemma 4.7]{CLY2011}, the following lemma holds.
\begin{lemma} Let $X\subset \Psi^+$ be a  mutually orthogonal root set. If  $X^{\rm cl}$ exists, then
$$e_{X^{\rm cl}}=\delta^{\#(X^{\rm cl}\setminus X)}e_{X}.$$
\end{lemma}

\section{An upper bound for the rank}

We introduce  notation for the following admissible sets of type $\ddF_4$ corresponding to those listed in Proposition \ref{prop: admset}.
\begin{eqnarray*}
X_0&=&\emptyset\\
X_1&=&\{\beta_4\}\\
X_2&=& \{\beta_1\}\\
X_3&=&\{\beta_3,\beta_3+2\beta_2\}\\
X_4&=&\{\beta_1,\beta_3\}^{\rm cl}\\
X_5&=&\{\beta_3, \beta_3+2\beta_2, \beta_3+2\beta_2+2\beta_1\}^{\rm cl}
\end{eqnarray*}

Let
\begin{enumerate}[(I)]
\item $N_0=W(\ddF_4)$, $A_0=\{1\}$, $C_0=N_0$,
\item $N_1=\left<r_1,\, r_2,\, r_{\epsilon_4+\epsilon_3},\, r_4\right>$
$C_1=\left<r_1,\, r_2,\, r_{\epsilon_4+\epsilon_3},\right>$, $A_1=\left<r_4\right>$,
\item $N_2=\left<r_1, \,r_3,\, r_4, r_{(\epsilon_1+\epsilon_2+\epsilon_3-\epsilon_4)/2} \right> $,
$C_2=\left<r_3,\, r_4\right>$,\\
$A_2=\left<r_1, r_{(\epsilon_1+\epsilon_2+\epsilon_3-\epsilon_4)/2},\, r_{(\epsilon_1+\epsilon_2-\epsilon_3+\epsilon_4)/2}, r_{(\epsilon_1-\epsilon_2+\epsilon_3+\epsilon_4)/2}\right>$,
\item $N_3=\left<r_2,\, r_3, \,r_{\epsilon_4},\, r_{\epsilon_4-\epsilon_1}\right>$,
$C_3=\left<r_{\epsilon_4-\epsilon_1}\right>$,
$A_3=\left<r_2,\, r_3,\, r_{\epsilon_4},\, r_{\epsilon_1}\right>$,
\item $N_4=\left<r_3,\, r_1,\,r_{(\epsilon_1+\epsilon_2+\epsilon_3-\epsilon_4)/2},\, r_{(\epsilon_1-\epsilon_2+\epsilon_3+\epsilon_4)/2}\right>$,
     $C_4=\{1\}$, $A_4=N_4$, 
\item  $N_5=\left<r_1, r_2, r_3, r_{\epsilon_4},r_{\epsilon_1}\right>$,
      $C_5=\{1\}$, $A_5=N_5$.
 \end{enumerate}
 The structure of these groups can be determined below.
\begin{lemma}\label{0to5}
 For  $N_i$, $A_i$, $C_i$,  $i=1$, $\ldots$, $5$, the following holds,
\begin{enumerate}[(I)]
\item $N_0\cong W(\ddF_4)$, $A_0\cong\{1\}$, $C_0\cong N_0$,
\item $N_1\cong W(\ddB_3)\times W(\ddA_1)$,
$C_1\cong W(\ddB_3)$, $A_2\cong  W(\ddA_1)$,
\item $N_2\cong W(\ddB_3)\times W(\ddA_1) $,
$C_2\cong W(\ddA_2)$,
$A_2\cong W(\ddA_1)^4$,
\item $N_3\cong W(\ddB_2)^2$,
$C_3\cong W(\ddA_1)$,
$A_3\cong W(\ddB_2)\times W(\ddA_1)^2$,
\item $A_4=N_4\cong W(\ddB_2)\times W(\ddA_1)^2$, $C_4=\{1\}$,
\item  $A_5=N_5\cong W(\ddB_3)\times W(\ddB_2)$, $C_5=\{1\}$.
 \end{enumerate}
\end{lemma}
\begin{proof}We do not give the full proof but restrict to  the case $i=1$. It can be checked  that
$\left<\epsilon_3+\epsilon_4, \beta_1\right>=-1$ and $\left<\epsilon_3+\epsilon_4, \beta_2\right>=0$, and
$\{\epsilon_3+\epsilon_4, \beta_1, \beta_2\}$ are linearly independent, hence $C_1\cong W(\ddB_3)$.
Since each  element of $\{\epsilon_3+\epsilon_4, \beta_1, \beta_2\}$ is orthogonal to $\beta_4$, so we
get that $N_1\cong W(\ddB_3)\times W(\ddA_1)$.
\end{proof}
 When we consider these groups in $\BrM(\ddF_4)$, the following lemma can be obtained.
 \begin{lemma}\label{NAC}
 For $i=0$, $\ldots$, $5$, the following holds.
 \begin{enumerate}[(I)]
 \item The group $N_i$ is the normalizer of $X_i$ in $W(\ddF_4)$.
 \item The group $N_i$ is the semidirect  product of $A_i$ and  $C_i$, with $C_i$ normalized by $A_i$.
 \item For $x\in A_i$,  we have $xe_{X_i}=e_{X_i}$.
 \item For $x\in C_i$, we have  $xe_{X_i}=e_{X_i}x $.
 \end{enumerate}
 \end{lemma}
 \begin{proof} Clearly, $N_i$ normalizes $X_i$, so $N_i\leq N(X_i)$ (the normalizer of $X_i$ in $W(\ddF_4)$), and the equality follows from  Lagrange's Theorem by verification in the table below.
Here  $\#N_i$ is known from
 Lemma \ref{0to5}, and the lengths of $W(\ddF_4)$-orbits are given in Proposition \ref{prop: admset}.
 Therefore the first claim hold.
 The proof of   the left conclusions is  as  arguments in   \cite[Section 6]{CLY2010} and \cite[Section 6]{CLY2011}.
 \end{proof}
 Suppose  $D_i$ is a set of left coset representatives of $N_i$ in $W(\ddF_4)$. We have the table below.
 In the table the product of the three entries in each row is equal to $1152$.
 \begin{center}
 \begin{tabular}{|c|c|c|c|}
                \hline
                $i$&$\#D_i$ & $\#C_i$&$\#A_i$\\
                \hline
                0 & 1 & 1152 &1\\
                1 & 12 & 48 &2\\
                2 & 12 & 6 & 16 \\
                3 & 18 & 2 &32\\
                4 & 36 & 1 &32\\
                5 & 3 & 1 &384 \\
                \hline
              \end{tabular}
 \end{center}
 We find that for the Brauer monoid  action of   some  monomial $\phi(a)$ for  $a\in \BrM(\ddF_4)$,
 the  admissible root sets  $\fp(\phi(a)\emptyset)$ and $\fp(\phi(a)^{\rm op}\emptyset)$ belong to  different $W(\ddF_4)$-orbits; for example $\fp(\phi(e_2e_3)\emptyset)=\{\beta_2\}$, and  $\fp(\phi(e_2e_3)^{\rm op}\emptyset)=\{\beta_3, \beta_2+\beta_3\}$.
 Some more groups as follows are needed.
 Let
\begin{eqnarray}
 N_6^L&=&\left<r_2,\, r_3,\, r_{\epsilon_4},\, r_{\epsilon_4-\epsilon_1}\right>,\\
 N_6^R&=&\left<r_2,\, r_{\epsilon_3},\, r_{\epsilon_4},\, r_{\epsilon_4-\epsilon_1}\right>,\\
 C_6&=&\left<r_{\epsilon_4-\epsilon_1}\right>,\\
 N_8&=&\left<r_2,\, r_4,\, r_{\epsilon_1}, r_{\epsilon_3}\right>.
 \end{eqnarray}
 Additionally, we choose $D_6^L$, $D_6^R$, and  $D_8$ to be sets of left coset representatives
  of $N_6^L$, $N_6^R$, and $N_8$ in $W(\ddF_4)$,
 respectively. \\
  Let $N_7^L=N_6^R$, $N_7^R=N_6^L$, $D_7^L=D_6^R$, $D_7^R=D_6^L$, $C_7=C_6$,\\
 $N_9^L=N_5$, $N_9^R=N_4$,  $D_9^L=D_5$, $D_9^R=D_4$,\\
 $N_{10}^L=N_4$, $N_{10}^R=N_5$,  $D_{10}^L=D_4$, $D_{10}^R=D_5$.\\
 In view of  \cite{CLY2010} and \cite{CLY2011}, the following lemma holds.
 \begin{lemma} \label{68}
 For  above groups, we have
 \begin{enumerate}[(I)]
 \item $N_6^L\cong W(\ddB_2)\times W(\ddB_2)$, $N_6^R\cong  W(\ddB_2)\times  W(\ddA_1)^2$,\\
 $C_6\cong W(\ddA_1)$, $N_8\cong W(\ddB_2)\times  W(\ddA_1)^2$.
 \item For each $a\in N_6^L$ $(b\in N_6^R)$, there exists some $c\in C_6$ such that\\
 $a e_3e_2= e_3e_2c $ $ (e_3e_2b= c e_3e_2) $.
 \item For each $a\in N_8$ we have that $a e_4r_3e_2e_3e_4=e_4r_3e_2e_3e_4$.
 \item For each $a\in N_9^L$  and $b\in N_9^R$, we have \\
 $a e_3e_2e_1e_3=  e_3e_2e_1e_3 $ and $ e_3e_2e_1e_3b=  e_3e_2e_1e_3 $.
 \end{enumerate}
 \end{lemma}
 \begin{thm}\label{rewrittenforms}
 Up to some power of $\delta$,  each monomial in $\BrM(\ddF_4)$ can be written in one of  the following forms.
\begin{enumerate}[(I)]
\item $u e_{X_i}v w$, $u\in D_i$, $w\in D_i^{-1}$, $v\in C_i$, $0\leq i\leq 5$.
\item  $u e_3e_2v w$, $u\in D_6^L$, $w\in (D_6^R)^{-1}$, $v\in C_6$.
\item  $u e_2e_3v w$, $u\in D_7^L$, $w\in (D_7^R)^{-1}$, $v\in C_7$.
\item  $u e_4r_3e_2e_3e_4 w$, $u\in D_8$, $w\in D_8^{-1}$.
\item $u e_3e_2e_1e_3 w$, $u\in D_9^L$, $w\in (D_9^R)^{-1}$.
\item $u e_3e_1e_2e_3 w$, $u\in D_{10}^L$, $w\in (D_{10}^R)^{-1}$.
\end{enumerate}
\end{thm}
\begin{proof}From \cite{CLY2011}, $(e_4r_3e_2e_3e_4)^{\rm op}=e_4r_3e_2e_3e_4$. In view of Proposition \ref{prop:opp},
it suffices to prove  that  the claim that the result of a left multiplication by each
 $r_i$ and
$e_{\beta}$  for $\beta\in \Psi^+$ at  the left of each element of $S$ can be written as   in (I)-(VI),
where
$$S=\{e_{X_i}\}_{i=0}^5\cup \{e_3e_2,\,e_2e_3, e_4r_3e_2e_3e_4,\, e_3e_2e_1e_3,\,  e_3e_1e_2e_3\}.$$
According to  \ref{NAC} and Lemma \ref{68},  the above holds for $\{r_i\}_{i=1}^4$.\\
By Proposition \ref{C3B3inF4}, and special cases $\Br(\ddC_3)$ in \cite{CLY2010} and $\Br(\ddB_3)$ in \cite{CLY2011},
we have that
$$\phi_1(\Br(\ddC_3))=\bigoplus_{s\in S_1} \Z[\delta^{\pm 1}] W(\ddC_3) s W(\ddC_3),$$
$$\phi_2(\Br(\ddB_3))=\bigoplus_{s\in S_2} \Z[\delta^{\pm 1}] W(\ddB_3) s W(\ddB_3),$$
 Where $$S_1=\{1, e_3, e_{X_2}, e_3e_2, e_2e_3, e_{X_3}, e_{X_4}, e_3e_2e_1e_3, e_1e_3e_2e_3, e_{X_5}\}$$
      $$S_2=\{1, e_{X_1}, e_2, e_3e_2, e_2e_3, e_{X_3}, e_4 e_4^{*}, e_4e_2, e_4r_3e_2e_3e_4\}$$
      with $e_4^{*}=r_3r_2r_3e_4r_3r_2r_3$.
      It can be seen that each element of $S_1$ and $S_2$ is in $S$ or conjugate to some element of $S$ under $W(\ddF_4)$.
      Therefore  the proof
      is reduced to cases of $\Br(\ddB_3)$ and $\Br(\ddC_3)$, which can be found in
      \cite{CLY2010} and \cite{CLY2011}.
\end{proof}

Aa a consequence, we obtain some information about
 rank of $\Br(\ddF_4)$ over $\Z[\delta^{\pm 1}]$.
\begin{cor} \label{upperrank}
As a $\Z[\delta^{\pm 1}]$-algebra, the algebra $\Br(\ddF_4)$ is spanned by $14985$ elements.
\end{cor}
\begin{proof}

 For $i=6$, $\ldots$, $10$, the following holds. \\
 \begin{center}
 \begin{tabular}{|c|c|c|c|}
                \hline
                set&cardinaltity & set&cardinality\\
                \hline
                $D_6^L$, $D_7^R$ & 18 & $D_6^R$, $D_7^L$ &36\\
               $C_6$, $C_7$ & 2 & $D_8$ &36\\
                $D_9^L$, $D_{10}^R$& 3 & $D_9^R$, $D_{10}^L$ & 36 \\
                \hline
              \end{tabular}
 \end{center}
By Theorem \ref{rewrittenforms} and numerical information from  the above two tables,
the algebra  $\Br(\ddF_4)$ over $\Z[\delta^{\pm 1}]$ has rank at most
 $$\sum_{i=0}^{5}(\#D_i)^2 \#C_i+2 \#D_6^L \#D_6^R \#C_6 +\#D_8^2 +2\#D_9^L \#D_9^R =14985.$$

\end{proof}
\section{$\phi(\Br(\ddF_4))$ in $\Br(\ddE_6)$}\label{sect:image}
We keep  notation as  in \cite[Section 2]{CW2011} and first introduce some basic concepts.
Let $M$ be  the  diagram  of a  connected finite simply laced Coxeter
group (type $\ddA$, $\ddD$, $\ddE_6$,  $\ddE_7$,  $\ddE_8$).
$\BrM(M)$ is the associated Brauer monoid as Definition \ref{1.1}.
An element  $a\in \BrM(M)$ is said to be of \emph{height} $t$ if the minimal number of
 $R_i$  occurring in an expression of $a$ is $t$, denoted by $\rm{ht}$$(a)$. By $B_Y$ we denote
the admissible closure (\cite{CFW2008}) of $\{\alpha_i|i\in Y\}$, where $Y$ is a coclique
of $M$. The set $B_Y$ is a minimal element in the $W(M)$-orbit of $B_Y$ which is endowed with a  poset  structure
induced by the partial ordering $<$
defined on $W(\ddE_6)$-orbits  in $\cA$ (the set of all admissible sets) in \cite{CGW2006}. If $d$ is the Hasse diagram distance for $W(M)B_Y$
from $B_Y$ to the unique maximal element
(\cite[Corollary 3.6]{CGW2006}), then for $B\in W(M)B_Y$ the height of $B$, already used in Definition
notation $\rm{ht}$$(B)$, is $d-l$, where $l$ is the distance in the
Hasse diagram from $B$ to the maximal element. \\
In \cite{CFW2008}, a Brauer  monoid action is defined as  follows.
For any mutually orthogonal positive root set $B$, we define $B^{\rm cl}$ to be  the \emph{admissible closure} of
$B$,  already used in Lemma \ref{lem:2roots} and  Definition \ref{defn:admclo} for a more difficult type,
 namely the minimal admissible root set(\cite{CFW2008}) containing $B$.
The generator $R_{i}$ acts
by the natural action of Coxeter group elements on its
root sets, where negative roots are negated so as to obtain positive roots,
the element $\delta$ acts as the identity,
and the action of $\{E_{i}\}_{i=1}^{n+1}$ is defined below.
\begin{equation}
E_i B :=\begin{cases}
B & \text{if}\ \alpha_i\in B, \\
(B\cup \{\alpha_{i}\})^{\rm cl} & \text{if}\ \alpha_i\perp B,\\
R_\beta R_i B & \text{if}\ \beta\in B\setminus \alpha_{i}^{\perp}.
\end{cases}
\end{equation}

By the natural involution, we can define a right monoid action of $\BrM(M)$ on $\cA$. \\
Considering our $\sigma$ and table 3  in \cite{CW2011},
let  $$Y\in \mathcal{Y}=\{\emptyset, \{2\}, \{1,6\}, \{2, 3, 5\} \}.$$
Obviously, each element of  $\mathcal{Y}$ is $\sigma$-invariant.
From \cite[Theorem 2.7]{CW2011}, it is known that each monomial  $a$ in $\BrM(\ddE_6)$ can be
uniquely written as $\delta^{i} a_{B} \hat{e}_Y h a_{B'}^{\rm op}$ for some $i\in \Z$ and $h\in W(M_{Y})$ in \cite[table 3]{CW2011},
where $B=a\emptyset$, $B^{'}=\emptyset a $, $a_{B}\in \BrM(\ddE_6)$, $a_{B'}^{\rm op}\in \BrM(\ddE_6)$ and
\begin{enumerate}[(i)]
\item $a\emptyset=a_{B}\emptyset=a_{B}B_Y$,  $\emptyset a= \emptyset a_{B'}^{\rm op}= B_Y a_{B'}^{\rm op}$,
\item $\rm{ht}$$(B)=$\rm{ht}$(a_{B})$, $\rm{ht}$$(B')=$\rm{ht}$(a_{B'}^{\rm op})$. 
\end{enumerate}
Now we can apply this to obtain the following corollary immediately.
\begin{cor}\label{lm:sigmainvariant} Let $B$, $Y$, $B'$ be as above.
Then   $a= a_{B} \hat{e}_Y h a_{B'}^{\rm op}$  is $\sigma$-invariant monomial in  $\BrM(\ddE_6)$ if and only if \\
$\rm(i)$ the sets $B$, $B^{'}$ are $\sigma$-invariant,  \\
$\rm(ii)$ the element $h\in W(M_Y)$ is $\sigma$-invariant.
\end{cor}
In type $\ddE_6$, we see that $E_2E_4E_5\{\alpha_6\}=\alpha_2$, $E_1E_3\{\alpha_4,\alpha_6\}=\{\alpha_1,\alpha_6\}.$
 Then
$$W(M_{\{2\}})=E_2E_4E_5W(M_{\{6\}})E_5E_4E_2,$$
$$W(M_{\{1,6\}})=E_1E_3W(M_{\{4,6\}})E_3E_1.$$
Let $\hat{E}_i=\delta^{-1}E_i$.
By the 6th column list of \cite[table 3]{CW2011},
 the group $W(M_{\{2\}})\cong W(\ddA_5)$ has generators $R_1\hat{E}_2$, $R_3\hat{E}_2$, $R_5\hat{E}_2$, $R_6\hat{E}_2$,
 $E_2E_4R_3E_5E_4\hat{E}_2$ and their relation is corresponding to  subdiagram of  $\ddE_6$ by deleting the node $2$ and the edge
 between $2$ and $4$ with
   $E_2E_4R_3E_5E_4\hat{E}_2$
 corresponding to node $4$, $R_i\hat{E}_2$ corresponding to node $i$ for $i=1$,$3$, $5$, $6$;
 the group $W(M_{\{1,6\}})\cong W(\ddA_2)$ has generators $R_2\hat{E}_1\hat{E}_6$,
 $R_4\hat{E}_1\hat{E}_6$ and whose relation is the subdiagram of $\ddE_6$
 of  nodes $2$, $4$ and the edge between them.
 When $\sigma$ acts on the generators of $W(M_{\{2\}})$ and $W(M_{\{1,6\}})$, we have that
 \begin{eqnarray*}
 \sigma(E_2E_4R_3E_5E_4\hat{E}_2)&=&E_2E_4R_3E_5E_4\hat{E}_2,\\
 \sigma(R_i\hat{E}_2)&=&R_{\sigma(i)}\hat{E}_2, \, \, (i=1, 3, 5, 6,)\\
 \sigma(R_j\hat{E}_1\hat{E}_6)&=&R_j\hat{E}_1\hat{E}_6, \,  \, (j=2, 4,)
 \end{eqnarray*}
which implies that the  $\sigma$-action on those two groups is determined the $\sigma$-action on the subdigrams of $\ddE_6$ described
in the above.
Hence $W(M_{\{2\}})^{\sigma}\cong W(\ddB_3)$, and $W(M_{\{1,6\}})^{\sigma}=W(M_{\{1,6\}})\cong W(\ddA_2)$. This conclusion can be summarized
 in the table below with the second column from our GAP \cite{GAP} code.
\begin{center}

\begin{tabular}{|c|c|c|c|}
  \hline
  $Y$ & $\#((W(\ddE_6)B_Y)^{\sigma})$ & $M_Y$ & $M_Y^{\sigma}$ \\
  \hline
  $\emptyset$ & $1$ & $\ddE_6$ & $\ddF_4$ \\
  $2$ & $12$ & $\ddA_5$ &  $\ddB_3$ \\
  $1$, $6$ & $30$ & $\ddA_2$ & $\ddA_2$ \\
  $2$, $3$, $5$ & 39 & $\emptyset$ & $\emptyset$ \\
  \hline
\end{tabular}
\end{center}

\begin{cor}\label{lowerrank}
The algebra $\Br(\ddE_6)^{\sigma}$ is free over $\Z[\delta^{\pm 1}]$ with  rank
$$1152+12^2\times 48 +30^2 \times 6 +39^2=14985.$$
\end{cor}
\begin{proof} By Corollary \ref{lm:sigmainvariant}, we have
\begin{eqnarray*}
\rm{rk}(\Br(\ddE_6)^{\sigma})=\Sigma_{Y\in \mathcal{Y}} (\#((W(\ddE_6)B_{Y})^\sigma))^2\#W(M_Y)^{\sigma}.
\end{eqnarray*}
 Hence the corollary holds.
\end{proof}
The following can be checked by computation and  application of  Theorem \ref{rewrittenforms}.
\begin{lemma} \label{surjective} Let $K_Y=\{a\in \BrM(\ddE_6)\mid \sigma(a)=a, a\emptyset\in W(\ddE_6)B_Y \}$ for $Y\in\mathcal{Y}$.
Then  $K_Y=\{\phi(b)\mid  b\in \BrM(\ddF_4),\, \phi(b)\emptyset \in W(\ddE_6)B_Y  \}$.
\end{lemma}
\begin{proof}  If $Y=\{2\}$, then  $\{\phi(b)\mid  b\in \BrM(\ddF_4),\, \phi(b)\emptyset \in W(\ddE_6)B_Y  \}$ are the elements
in (I) for $i=1$ in Theorem \ref{rewrittenforms}, which has cardinality $12^2\times 48$, and $\hE_2\phi(C_1)\hE_2=W(M_Y)^{\sigma}$ (in the proof of Corollary  \ref{lm:sigmainvariant}), and the  $\phi(W(\ddF_4))$-orbit of  $\{\alpha_2\}$
 is  $\sigma$-invariant elements in $W(\ddE_6)B_Y$. Hence the lemma  holds for $Y=\{2\}$.\\
If $Y=\{1,6\}$, then  $\{\phi(b)\mid  b\in \BrM(\ddF_4),\, \phi(b)\emptyset \in W(\ddE_6)B_Y  \}$ are those monomials in (I) for $i=2$ and $i=3$, (II),
(III) and  (IV) in Theorem \ref{rewrittenforms}.
Analogous to the above argument for $Y=\{2\}$, we see that the image of those monomials under $\phi$ correspond to different normal forms   for  $\BrM(\ddE_6)$ up to some powers of $\delta$, and  those monomials of
  in (I) for $i=2$ and $i=3$, (II),
(III) and  (IV) in Theorem \ref{rewrittenforms} have cardinality
$12^2\times 6 +18^2\times 2 +18\times 36\times 2 +18\times 36\times 2 +36^2=30^2\times 6$. Hence the lemma holds for $Y=\{1,6\}$. \\
If $Y=\{2,3, 5\}$, then  $\{\phi(b)\mid  b\in \BrM(\ddF_4),\, \phi(b)\emptyset \in W(\ddE_6)B_Y  \}$ are those monomials in (I) for $i=4$ and $i=5$, (V),
and (VI) in Theorem \ref{rewrittenforms}.
Similarly, we see those monomials are corresponding to different normal forms  for $\BrM(\ddE_6)$ up to some powers of $\delta$, and  they have cardinality
$36^2+3^2+36\times 3+36\times 3=39^2$. Hence the lemma holds for $Y=\{2,3,5\}$.
\end{proof}
Now, we can give the  proof of our main Theorem \ref{mainthm}.
\begin{proof}  Proposition \ref{phibeinghomo} implies $\phi$ is a homomorphism.  Corollary \ref{upperrank} indicates that the $\phi(\Br(\ddF_4))$
has rank at most $14985$. Corollary \ref{lowerrank} implies that the $\Br(\ddE_6)^{\sigma}$ has rank $14985$.
Lemma \ref{surjective} indicates that $\phi$ has image $\SBr(\ddE_6)$; therefore $\phi$ is surjective. The homomorphism $\phi$ is an isomorphism and $\Br(\ddF_4)$ is free over $\Z[\delta^{\pm 1}]$
of rank $14985$ because of  the freeness of  $\SBr(\ddE_6)$.
 \end{proof}
\section{Cellularity}
 Recall from \cite{Gra} and \cite{GL1996} that an associative
algebra $\alg$ over a commutative ring $R$ is cellular if there is a quadruple
$(\Lambda, T, C, *)$ satisfying the following three conditions.

\begin{itemize}
\item[(C1)] $\Lambda$ is a finite partially ordered set.  Associated to each
$\lambda \in \Lambda$, there is a finite set $T(\lambda)$.  Also, $C$ is an
injective map
$$ \coprod_{\lambda\in \Lambda} T(\lambda)\times T(\lambda) \rightarrow \alg$$
whose image is an $R$-basis of $\alg$.

\item[(C2)]
The map $*:\alg\rightarrow \alg$ is an
$R$-linear anti-involution such that
$C(x,y)^*=C(y,x)$ whenever $x,y\in
T(\lambda)$ for some $\lambda\in \Lambda$.

\item[(C3)] If $\lambda \in \Lambda$ and $x,y\in T(\lambda)$, then, for any
element $a\in \alg$,
$$aC(x,y) \equiv \sum_{u\in T(\lambda)} r_a(u,x)C(u,y) \
\ \ {\rm mod} \ \alg_{<\lambda},$$ where $r_a(u,x)\in R$ is independent of $y$
and where $\alg_{<\lambda}$ is the $R$-submodule of $\alg$ spanned by $\{
C(x',y')\mid x',y'\in T(\mu)\mbox{ for } \mu <\lambda\}$.
\end{itemize}
Such a quadruple $(\Lambda, T, C, *)$ is called a {\em cell datum} for
$\alg$.\\
There is also an equivalent definition due to K\"oning and Xi.
\begin{defn}\label{defn:KXcelluar}
 Let $A$ be $R$-algebra. Assume there is an anti-automorphism $i$ on $A$ with $i^2=id$. A two sided ideal $J$ in
$A$ is called cellular if and only if $i(J)=J$ and there exists a left ideal $\Delta\subset J$ such that $\Delta$ has finite rank and 
there is an isomorphism of $A$-bimodules $\alpha:J\simeq \Delta\otimes_R i(\Delta)$ making the following diagram commutative:
\begin{center}
\xymatrix{
J\ar[rr]^{\alpha}\ar[d]_{i} &   & \Delta\otimes_R i(\Delta) \ar[d]^{x\otimes y\rightarrow i(y)\otimes i(x)} \\
          J \ar[rr]^{\alpha}          & & \Delta\otimes_R i(\Delta)
          }
\end{center}
The algebra $A$ is called \emph{cellular} if  there is a vector space decomposition $A=J'_1\oplus \cdots\oplus J'_n$ with
$i(J'_j)=J'_j$ for each $j$ and such that setting $J_j=\oplus_{k=1}^jJ'_j$ gives a chain of two sided ideals of
$A$ such that for each $j$ the quotient $J'_j=J_j/J_{j-1}$  is a cellular ideal of $A/J_{j-1}$.
\end{defn}
Also recall definitions of iterated inflations from \cite{KX1999} and  cellularly stratified algebra from \cite{HHKP2010} in Definition \ref{defn:cellstrat}.
 Given an $R$-algebra $B$, a finitely generated free $R$-module $V$, and a bilinear form
 $\varphi: V\otimes_{R}V\longrightarrow  B$ with values in $B$, we define an  associative algebra (possibly without unit)
 $A(B, V, \varphi)$ as follows: as an $R$-module, $A(B, V, \varphi)$ equals $V\otimes_{R}V\otimes_{R}B$. The multiplication is defined on basis
 element as follows:
 \begin{eqnarray*}
 (a\otimes b \otimes x)(c\otimes d \otimes y):=a\otimes d \otimes x\varphi(b,c)y.
\end{eqnarray*}
Assume that there is an involution $i$ on $B$. Assume, moreover, that $i(\varphi(v,w))=\varphi(w,v)$.
If we can extend this involution $i$ to $A(B, V, \varphi)$ by defining $i(a\otimes b \otimes x)=b\otimes a \otimes i(x)$. Then
We call $A(B, V, \varphi)$ is an \emph{inflation} of $B$ along $V$.
Let $B$ be an inflated algebra (possible without unit) and $C$ be an algebra with unit. We define an algebra structure in such a way  that
$B$ is a two-sided ideal and $A/B=C$. We require that $B$ is an ideal, the multiplication is associative, and
that there exists a unit element of $A$ which maps onto the unit of the quotient $C$. The necessary conditions are outlined in
\cite[Section 3]{KX1999}. Then we call $A$  an inflation of $C$ along $B$, or iterated inflation of $C$ along $B$.
We present  Proposition 3.5 and Theorem 4.1 of \cite{KX1999}.
\begin{prop} An inflation of a cellular algebra is cellular again. In particular, an iterated inflation of $n$ copies of
$R$ is cellular, with a cell chain of length $n$ as in Definition \ref{defn:KXcelluar}.
\end{prop}
More precisely, the second statement has the following meaning. Start with $C$ a full matrix ring over $R$ and  $B$ an inflation of $R$ along a
free $R$-module, and from a new $A$ which is an inflation of the old $A$ along the new $B$, and continue this operation. Then after
$n$ steps we have produced a cellular algebra $A$ with a cell chain of length $n$.
\begin{thm}Any cellular algebra over $R$ is the iterated inflation of finitely many copies of $R$. Conversely, any iterated inflation of
finitely many copies of $R$ is cellular.
\end{thm}
Let $A$ be  cellular(with identity) which can be realized as an iterated inflation of cellular algebras $B_l$ along vector spaces $V_l$
for $l=1, \ldots,n.$ This implies that as a vector space
\begin{eqnarray*}
A=\oplus_{l=1}^{n} V_l\otimes V_l\otimes B_l,
 \end{eqnarray*}
 and $A$ is cellular with a chain of two sided ideals ${0}=J_0\subset J_1\cdots \subset J_n=A$, which can be refined to a cell chain, and each
 quotient $J_l/J_{l-1}$ equals $V_l\otimes V_l\otimes B_l$ as an algebra without unit. The involution $i$ of $A$,is defined through the involution $i_l$ of the algebra $B_l$ where
 $i(a\otimes b\otimes x)=b\otimes a\otimes j_l(x)$. The multiplication rule of a layer $V_l\oplus V_l\oplus B_l$ is indicated by
 \begin{eqnarray*}
 (a\otimes b \otimes x)(c\otimes d \otimes y):=a\otimes d \otimes x\varphi(b,c)y+\text{lower terms}.
\end{eqnarray*}
 Here lower terms refers to element in lower layers $V_h\otimes V_h\otimes B_h$ for $h<l$.
 Let $1_{B_l}$ be the identity of the algebra $B_l$.\\
 Let that $R$ is a field.
 \begin{defn}\label{defn:cellstrat}
  A finite dimensional associative algebra $A$ is called cellularly stratified with stratification data
 $(B_1, V_1,\ldots, B_n, V_n)$ if and only if the following conditions are satisfied:
  \begin{itemize}
\item[(1)] The algebra is an iterated inflation of cellular algebra $B_l$ along vector spaces $V_l$ for $l=1$,$\ldots$, $n$.
\item[(2)]For each $l=1$,$\ldots$, $n$, there exist $u_l$, $v_l$
such that $e_l=u_l\otimes v_l\otimes 1_{B_l}$ is an idempotent.
\item[(3)] If $l>m$, then $e_le_m=e_m=e_me_l$.
\end{itemize}
 \end{defn}
 As \cite{BO2011}, the following theorem can be obtained.
 \begin{thm} Let $R$ be a field with  $2$ and $3$ being invertible in $R$ and containing   $\Z[\delta^{\pm 1}]$ as a subring.
 Then the algebra $\Br(R,\ddF_4)=\Br(\ddF_4)\otimes_{\Z[\delta^{\pm 1}]} R$ is a cellularly stratified algebra over $R$.
 \end{thm}
 \begin{proof} To prove this, it suffices to prove it for  $\Br(R,\ddE_6)=\SBr(\ddE_6)\otimes_{\Z[\delta^{\pm 1}]} R$ because of
 Theorem \ref{mainthm}.
 Let $Z_0\subset Z_1\subset Z_2\subset Z_3$ be a $\sigma$-invariant and  admissible root set sequence of type $\ddE_4$, where
  $Z_0=\emptyset$, $Z_1=\{\alpha_2\}$, $Z_2=\{\alpha_2, \alpha_2+\alpha_3+\alpha_5+2\alpha_4\}$,
  $Z_3=\{\alpha_2,\alpha_3,\alpha_5, \alpha_2+\alpha_3+\alpha_5+2\alpha_4\}$.
  As $E_2E_4E_5E_3\{\alpha_1, \alpha_6\}=Z_2$, we have $\rm ht$$(Z_2)=0$. \\
 For $0\leq i\leq 3$, let $B_{Z_i}$ be the group algebras of $W_{Z_0}=W(M_{\emptyset}^\sigma)$, $W_{Z_1}=W(M_{\{2\}}^{\sigma})$,
 $W_{Z_2}=E_2E_4E_5E_3W(M_{\{1,6\}}^{\sigma})E_3E_5E_4E_2$, $W_{Z_3}=W(M_{\{2,3,5\}}^{\sigma})$ over $R$, respectively, whose group rings over $R$
 are cellular algebras due to \cite[Theorem 1.1]{G2007}.\\
Then  each monomial  $a$ in $\BrM(\ddE_6)^\sigma$ can be
uniquely written as $\delta^{i} a_{Z_i, B} \hat{e}_Y h a_{Z_i, B'}^{\rm op}$ for some $i\in\{0,1,2,3\}$ and $h$ is from the above four groups,
where $B=a\emptyset$, $B^{'}=\emptyset a $ being $\sigma$-invariant, $a_{Z_i,B}\in \BrM(\ddE_6)^{\sigma}$, $a_{Z_i,B'}^{\rm op}\in \BrM(\ddE_6)^{\sigma}$
and
\begin{enumerate}[(i)]
\item $a\emptyset=a_{Z_i,B}\emptyset=a_{B}Z_i$,  $\emptyset a= \emptyset a_{Z_i, B'}^{\rm op}= Z_i a_{B'}^{\rm op}$,
\item $\rm{ht}$$(B)=$\rm{ht}$(a_{Z_i,B})$, $\rm{ht}$$(B')=$\rm{ht}$(a_{Z_i, B'}^{\rm op})$. 
\end{enumerate}
For each $Z_i$,   let $V_{Z_i}$ be a linear space over $R$ with basis $u_{Z_i, B}$ where $B\in W(\ddE_6)Z_i$ and $\sigma(B)=B$.  and let $\varphi_{Z_i}$
be a bilinear map defined as
\begin{eqnarray*}
V_{Z_i}\otimes_{R}V_{Z_i}&\longrightarrow& B_{Z_i}\\
\varphi_{Z_i}(u_{Z_i, B},u_{Z_i, B'})&=&a_{Z_i, B}^{\rm op}a_{Z_i, B'}, \text{  }\, if\, \text{  }  Z_i=a_{Z_i, B}^{\rm op}a_{Z_i, B'}\emptyset,\\
\varphi_{Z_i}(u_{Z_i, B},u_{Z_i, B'})&=&0,  \text{  }\,if\text{  }\, Z_i\subsetneqq  a_{Z_i, B}^{\rm op}a_{Z_i, B'}\emptyset.
\end{eqnarray*}
We first prove that $\varphi$ is well defined.
As $a_{Z_i, B}^{\rm op}\empty=Z_i$, we find $Z_i\subset a_{Z_i, B}^{\rm op}a_{Z_i, B'}\emptyset$, similarly
$Z_i\subset \emptyset a_{Z_i, B}^{\rm op}a_{Z_i, B'}$. If $Z_i=a_{Z_i, B}^{\rm op}a_{Z_i, B'}\emptyset$, this indicates that
$Z_i=\emptyset a_{Z_i, B}^{\rm op}a_{Z_i, B'}$ and that $a_{Z_i, B}^{\rm op}a_{Z_i, B'}$ will be in $W_{Z_i}$ up to some power of
$\delta$. Therefore our $\varphi_{Z_i}$ is well defined. Observe that
$$(a_{Z_i, B}^{\rm op}a_{Z_i, B'})^{\rm op}=a_{Z_i, B'}^{\rm op}(a_{Z_i, B}^{\rm op})^{\rm op}=a_{Z_i, B'}^{\rm op}a_{Z_i, B},$$
so $(\varphi_{Z_i}(u_{Z_i, B},u_{Z_i, B'}))^{\rm op}=\varphi_{Z_i}(u_{Z_i, B'},u_{Z_i, B})$. By linearly extension, we find
$(\varphi_{Z_i}(u,v))^{\rm op}=\varphi_{Z_i}(v,u)$, for $u$,$v\in V_{Z_i}$.
By the proof of Lemma \ref{surjective},  the algebra $\SBr(R,\ddE_6)\otimes_{Z[\delta^{\pm 1}]}$ is an iterated inflation of cellular algebra $B_{Z_i}$ along vector space
$V_{Z_i}$ for $Z_0$, $\ldots$, $Z_3$, namely $\SBr(R,\ddE_6)\otimes_{Z[\delta^{\pm 1}]}$ satisfies (1) of cellulary stratified algebra.
We take $e_{Z_i}=u_{Z_i, Z_i}\otimes u_{Z_i, Z_i}\otimes 1_{B_{Z_i}} $, where  $1_{B_{Z_i}}=\delta^{-\#Z_i} E_{Z_i}$. Because $E_{Z_i}E_{Z_j}=\delta^{\#Z_i}E_{Z_j}$ for $Z_i\subset Z_j$,  hence the condition (2) and (3) follows since that
that $Z_i>Z_j$ means $Z_i\subsetneqq Z_j$. Finally, $\SBr(R,\ddE_6)\otimes_{Z[\delta^{\pm 1}]}$ is a cellularly stratified algebra.
 \end{proof}

Shoumin Liu \\
Eindhoven University of Technology\\
liushoumin2003@gmail.com 

\end{document}